\documentclass[12pt,bezier,]{article}%draft
\usepackage{graphicx}
\usepackage{epstopdf}
\usepackage{subfigure,booktabs}
\usepackage{enumerate}
\usepackage{fourier}%统一修改正文和数学字体为Adobe Utopia
\usepackage[colorlinks,anchorcolor=blue,linkcolor=blue,citecolor=blue]{hyperref}
\usepackage[all]{hypcap}%引用链接会跳转至浮动体开始的地方
\usepackage[fleqn]{amsmath}%公式左对齐
\usepackage[numbers,sort&compress]{natbib}%压缩并排序参考文献标号
\usepackage{amsthm,amsfonts,amssymb,mathrsfs,bbding}%调用定理
\usepackage{amssymb}
\usepackage{float}%图片插入指定位置
\evensidemargin=\oddsidemargin\topmargin=-1.5cm

\newtheorem{lem}{Lemma}[section]
\newtheorem{cor}[lem]{Corollary}
\newtheorem{thm}[lem]{Theorem}
\newtheorem{defin}[lem]{Definition}

\newtheorem{prob}{Problem}
\newtheorem{conj}{Conjecture}
\newtheorem{claim}{Claim}

\newtheorem{remark}{Remark}[section]

\newtheorem*{proof0}{Proof of Theorem 1.5}
\newtheorem*{proof1}{Proof of Theorem 1.9}
\newtheorem*{proof2}{Proof of Corollary 1.12}

\usepackage{CJK}

\begin{document}
\title{$r$-vertex-strongly-distinguishing total coloring of graphs\thanks{This work is supported by National Natural Science Foundation of China (Grant Nos. 11961041, 61802158), and the Foundation of China Scholarship Council (No. 201908620009).}}

\author{{\small Fei Wen$^{1,4,}$\footnote{Corresponding author. {E-mail addresses:}
wenfei@lzjtu.edu.cn(F. Wen).},\ \ \ Zepeng Li$^{2}$,\ \ \ Xiang'en Chen$^{3}$}\\[2mm]
\scriptsize 1.  Institute of Applied Mathematics, Lanzhou Jiaotong
University, Lanzhou 730070, P.R.China\\
\scriptsize 2. School of Information Science and Engineering, Lanzhou University, Lanzhou 730000, P.R.China\\
\scriptsize 3. College of Mathematics and Statistics, Northwest Normal University, Lanzhou 730070, P.R.China\\}
\date{}
\maketitle
{\flushleft\large\bf Abstract.} Inspired by the phenomenon of \emph{co-channel interference} in communication network, a novel graph parameter, called \emph{$r$-vertex-strongly-distinguishing total coloring} (abbreviate as $D(r)$-VSDTC), is proposed in this paper. Given a graph $G$, an $r$-VSDTC is an assignment of $k$ colors to $V(G)\cup E(G)$ such that any two adjacent or incident elements receive different colors and any two vertices with distance at most $r$ have distinct color-set, where the color-set of a vertex $u$ is the set of colors assigned on $u$ and its neighborhoods and incident edges. The \emph{$r$-vertex-strongly-distinguishing total chromatic number} of $G$, denoted by $\chi_{r-vsdt}(G)$, is the minimum integer $k$ for which $G$ admits a $k$-$D(r)$-VSDTC. We show that $\chi_{1-vsdt}(G)\leq 4\Delta(G)$ for every graph $G$ without isolated edges and $\chi_{1-vsdt}(G)\le k\Delta(G)+3$ for a $k$-degenerated graph $G$ without isolated edges, where $1\le k\le 3$.

{\flushleft\large\bf Keywords:} $k$--degenerated graph; $r$-vertex-strongly-distinguishing total coloring; $r$-vertex-strongly-distinguishing total chromatic number
 \vspace{1mm}

\maketitle{\flushleft\large\bf AMS Classifications} 05C15

\section{Introduction}
Let $G=(V,E)$ be a finite, simple and undirected graph with vertex set $V(G)$ and edge set $E(G)$. For any vertex $u,v\in V(G)$, the \emph{distance} between $u$ and $v$, denoted by $d(u,v)$, is defined as the length of a shortest path between them. The \emph{diameter} of $G$, denoted by $\mathbf{d}$, is the maximum distance between any two vertices of $G$. The \emph{incident set} of the vertex $u$, denoted by $N\langle u\rangle$, is a set consisting of $N[u]$ and incident edges of $u$, i.e., $N\langle u\rangle= N[u]\cup\{uv|v\in N(u)\}$, where $N[u]$ and $N(u)$ are the \emph{closed neighbor set} and \emph{open neighbor set} of $u$, respectively. In addition, $d(u)$ represents the \emph{degree} of $u$.

As usual, let $\Delta(G)$ and $\delta(G)$ denote the \emph{maximum degree} and the \emph{minimum degree} of $G$, respectively. If there is no
confusion, sometimes we write $\Delta(G)$ as $\Delta$, and write $\delta(G)$ as $\delta$. And we always denote respectively, $K_{n}$ and $P_{n}$ by the \emph{complete graph} and \emph{path} of order $n$. For a universal set $I$, suppose $A, B\subseteq I$, the symbols $A-B=A\setminus (A\cap B)$ and $|A|$ respectively represent the \emph{minus set} and the \emph{cardinal number} of $A$. For the other terminologies and notations refer to \cite{bondy}.

Graph coloring theory is becoming an increasingly useful
family of mathematical models for a broad range of applications,
such as \emph{register allocation}, \emph{frequency assignment}, \emph{time tabling and scheduling} and so on. A \emph{$\kappa$-vertex coloring} of $G$ is an assignment of $\kappa$ colors to the vertices of $G$ such that any two adjacent vertices have different colors, and the \emph{chromatic number} of $G$, denote by $\chi(G)$, is the minimum $\kappa$ for which $G$ admits a $\kappa$-coloring.

In 1997, Burris and Schelp\cite{Burris} first introduced  \emph{vertex distinguishing edge-coloring} of graphs. After that, \emph{(adjacent-) vertex distinguishing edge coloring} of graphs was proposed and has been drawn a lot of attention of the researchers (see Ref.\cite{zhang-4,hatami,balister}). In 2006, Akbari et al.\cite{limc} proposed \emph{$r$-strong edge coloring of graphs} in the following.

\begin{defin}[\!\!\cite{limc}]\label{def-2}
Let $G$ be a graph, and $f$ a \emph{proper $\kappa$-total coloring} from $E(G)$ to $\{1,2,\ldots,$ $\kappa\}$. If $C_{f}(u) \not=C_{f}(v)$ for any two vertices $u$ and $v$ with distance at most $r$, then $f$ is called \emph{$\kappa$-$r$-strong edge coloring} (denoted by $\kappa$-$D(r)$-SEC for short) where $C_{f}(u)=\{f(uv)|uv\in E(G)\}$, and the \emph{$r$-strong edge chromatic number} $\chi'_{s}(G, r)$ is the minimum number of colors required for such a coloring of $G$.
\end{defin}

Moreover, they gave an upper bound of \emph{$1$-strong edge chromatic number} of a graph, that is, for any graph $G$ with no isolated edges, $\chi'_{s}(G,1)\le 3\Delta(G)$. Moreover, they also obtained \emph{$2$-strong edge chromatic number} and \emph{$3$-strong edge chromatic number} of a tree, which generalized Zhang's result on \emph{$2$-strong edge coloring} of a tree\cite{zhang-4}. In 2005, Zhang et al.\cite{zhang-1} presented a concept of the \emph{adjacent-vertex-distinguishing total coloring of graphs} and, meanwhile, obtained the adjacent-vertex-distinguishing total chromatic number of some graphs such as \emph{cycle, complete graph, complete bipartite graph, fan, wheel and tree}.

After then, Zhang et al.\cite{zhang-2} defined the \emph{$D(\beta)$-vertex-distinguishing total coloring} (abbreviated as $D(\beta)$-VDTC for short) of graph $G$, that is a proper \emph{total coloring} such that for any vertex $u,v\in V(G)$ with  $1\leq d(u,v)\leq \beta$, and $C(u)\neq C(v)$ always holds, where $C(u)$ is the color set of the edges incident to $u$ together with the color on $u$, and $\chi_{\beta-vdt}(G)$ is called the \emph{$D(\beta)$-vertex-distinguishing total chromatic number} for a $D(\beta)$--VDTC of $G$. They subsequently investigated the chromatic number $\chi_{\beta-vdt}(G)$ on some familiar graphs (such as paths, cycles, complete graphs and so on) when $\beta=2,3$.

%After then, Zhang et al.\cite{zhang-3} also introduced \emph{adjacent-vertex-strongly-distinguishing total coloring of graph}, that is, an \emph{adjacent-vertex-strongly-distinguishing total coloring} is a proper total coloring of graph $G$ if for any $uv\in E(G)$, the set of colors assigned to $u$ together with its incident elements differs from that of colors assigned to $v$ together with its incident elements. The minimum number of colors required for the coloring of $G$ is denoted by $\chi_{asdt}(G)$. Meanwhile, Zhang et al. conjectured that for any graph $G$ of order $n$, $\chi_{asdt}(G)\leq n+\lceil\log_{2}n \rceil+1$, and showed that the conjecture holds for some simple graphs such as paths, cycles, complete bipartite graphs and a tree.

%上面这一段建议删掉，你考虑一下

In \emph{communication network}, if two nodes within certain distance send out the same frequency signals, then it would be appeared the phenomenon of \emph{interference} since the same frequency waves have resonated\cite{Liang}. In some literature\cite{winter}, which is also called the \emph{co-channel interference}. Especially, if the frequency of  the two nodes that send out signals, and the signals as well as the nodes that receive signals are identical, then the phenomenon of interference are more serious. So, combine with the above physical phenomenon and the nice results of \cite{limc,zhang-3}, we defined a new concept entitled \emph{$r$-vertex-strongly-distinguishing total coloring} of graph in the following.
\begin{defin}\label{def-1}
Let $G$ be a graph with no isolated edges as its components, and let $f$ be a \emph{proper $\kappa$-total coloring} from $V(G)\cup E(G)$ to $\{1,2,\ldots, \kappa\}$. We call $f$ an \emph{$\kappa$-$r$-vertex-strongly-distinguishing total coloring} (denoted by $\kappa$-$D(r)$-VSDTC for simplicity) if any two vertices $u$ and $v$ with distance at most $r$ meet $C_{f}\langle u\rangle \not=C_{f}\langle v\rangle$, where $C_{f}\langle u\rangle=\{f(x)|x\in N\langle u\rangle\}$. The $r$-vertex-strongly-distinguishing total chromatic number of $G$ is the minimum $\kappa$ for which $G$ admits an \emph{$\kappa$-$D(r)$-VSDTC}, denote by $\chi_{r-vsdt}(G)$ for short.
\end{defin}
Particularly, if $r=1$ or $r=\mathbf{d}$, it is respectively called to be the \emph{adjacent-vertex-strongly-distinguishing total coloring} and \emph{vertex-strongly-distinguishing total coloring} of $G$, those concepts have been introduced by Zhang et al. in \cite{zhang-3}, but the latter one has no any research till now. Clearly, if $G$ admits an $r$-vertex-strongly-distinguishing total coloring, it is easy to see that $\chi_{1-vsdt}(G)\le \chi_{r-vsdt}(G)\le \chi_{\mathbf{d}-vsdt}(G)$.

\emph{Determining the bound of vertex-distinguishing colorings} of a graph is an important and exciting research topic in mathematics science and theoretical computer science. The \emph{vertex-strongly-distinguishing chromatic numbers} of a graph provide information on its structural properties and also on some relevant topological properties, in particular those related to information science. In this paper, we mainly consider \emph{$r$-vertex-strongly-distinguishing total coloring} of some graphs, especially $k$-degenerated graphs $k=1,2,3$. At first, we by Definition \ref{def-1} give two obvious results of \emph{$r$-vertex-strongly-distinguishing total coloring} of a graph.
\begin{thm}\label{thm-1}
Let $G$ be a connected graph with order no less than $3$, and $r_{i}$ the positive numbers. If $r_{1}<r_{2}$, then $$\chi_{r_{1}-vsdt}(G)\le \chi_{r_{2}-svdt}(G).$$
\end{thm}
\begin{thm}\label{thm-2}
Let $G$ be a graph with $m$ connected components $G_{1},G_{2},\ldots,G_{m}$. Then
$$\chi_{r-vsdt}(G)=\max\{\chi_{r-vsdt}(G_{i})|i=1,2,\ldots,m\}.$$
\end{thm}
\begin{remark}\label{remark-1}
Theorem \ref{thm-2} implies that if one would like to consider \emph{$r$-vertex-strongly-distinguishing total coloring} of a graph $G$, it only needs to verify the connected component with maximum chromatic number of such coloring.
\end{remark}

An \emph{$r$-vertex-strongly-distinguishing total coloring} of a graph can be obtained from  an \emph{$r$-strong edge coloring} and a \emph{proper vertex-coloring} of the graph, so one can get the following result.
\begin{thm}\label{thm-3}
For any graph $G$ with no isolated edges,
$$\chi_{r-vsdt}(G)\le \chi'_{s}(G,r)+\chi(G).$$
\end{thm}

In \cite{limc}, Akbari et al. obtained an upper bound of $1$-strong edge chromatic number of a graph $G$.

\begin{lem}[\!\cite{limc}]\label{lem-0}
For any graph $G$ with no isolated edges, $\chi'_{s}(G,1)=3\Delta(G)$.
\end{lem}

By Lemma \ref{lem-0} and \emph{Brook's Theorem}, we obtain the following result.

\begin{cor}\label{cor-1}
For any graph $G$ with no isolated edges,
$$\chi_{1-vsdt}(G)\le 4\Delta(G).$$
\end{cor}
\begin{remark}\label{remark-2}
Hatami\cite{hatami} proved that $\chi'_{s}(G,1)\le \Delta(G)+300$, provided that  $\Delta(G)>10^{20}$. By Theorem \ref{thm-3} one can also get $\chi_{1-vsdt}(G)\le 2\Delta(G)+300$ when $\Delta(G)>10^{20}$. In fact, this bound for familiar graphs is invalid. In \cite{balister}, Balister et al. proved $\chi'_{s}(G,1)\le 5$ for such graphs with maximum degree $\Delta(G)=3$, and prove $\chi'_{s}(G,1)\le \Delta(G)+2$ for bipartite graphs. Moreover, for $k$-chromatic graphs $G$ without isolated edges they prove a weaker result of the form $\chi'_{s}(G,1)\le \Delta(G)+O(\log k)$. Also by Theorem \ref{thm-3}, $\chi_{1-vsdt}(G)\le 8$ for graphs with maximum degree $\Delta(G)=3$, and $\chi_{1-vsdt}(G)\le \Delta(G)+4$ for bipartite graphs, and for $k$-chromatic graphs $G$ without isolated edges, one can obtain form $\chi_{1-vsdt}(G)\le \Delta(G)+k+O(\log k)$.
\end{remark}

A graph $G$ is \emph{$k$-degenerated} if $\delta(H)\le k$ for every subgraph $H$
of $G$. In literature \cite{zhang-4}, Zhang et al. proved that for any tree $T$, $\chi_{s}'(T,1)\le \Delta (T)+1$. Combine with Theorems \ref{thm-2}, \ref{thm-3} and \emph{Brook's Theorem} we have the following corollary.
\begin{cor}\label{cor-3}
Let $G$ be a $1$-degenerated graph. Then
$$\chi_{1-vsdt}(G)\le \Delta(G)+3.$$
\end{cor}
Next, we give an upper bound of $\chi_{1-vsdt}(G)$ for $k$-degenerated graphs $G$, where $k\ge 2$.

\begin{thm}\label{thm-4}
For $k\ge 2$, let $G$ be a $k$-degenerated graph without isolated edges. Then
$$\chi_{1-vsdt}(G)\le  k\Delta(G)+3.$$
\end{thm}

Together with Theorems \ref{thm-3} and \ref{thm-4}, one can obtain the following corollary.
\begin{cor}\label{cor-2}
For $k\ge 2$, let $G$ be a $k$-degenerated graph without isolated edges. Then
$$\chi_{1-vsdt}(G)\le \min\{k\Delta(G)+3,\ \ 4\Delta(G)\}.$$
\end{cor}

From Corollaries \ref{cor-3} and \ref{cor-2}, we give an improved upper bound of $\chi_{1-vsdt}(G)$ for $k$-degenerated graphs $G$ with $k=1,2,3$.
\begin{thm}\label{thm-5}
Let $G$ be a $k$-degenerated graph with $k=1,2,3$ and without any isolated edge. Then
$$\chi_{1-vsdt}(G)\le  k\Delta(G)+3.$$
\end{thm}

In light of Akbari et al.'s results\cite{limc}  about a tree, we also give an upper bound of \emph{$2$-vertex-strongly-distinguishing total chromatic number} of a tree.
\begin{cor}\label{pro-1}
Let $T$ be any tree with $\Delta(\geq2)$. Then
$\chi_{2-svdt}(T)\leq \Delta(T)+3.$ Moreover, if for any two vertices $u$ and $v$ with maximum degree, and $d(u,v)\ge3$, then $\chi_{2-svdt}(T)=\Delta(T)+2$.
\end{cor}

From Lemma \ref{lem-6} we know that any tree $T$ with $\Delta\ge 3$ admits a $(2\Delta-1)$-$D(3)$-VSDTC, applying the same method as the proof of Corollary \ref{pro-1}, one can obtain the following corollary .
\begin{cor}\label{pro-2}
Let $T$ be any tree with $\Delta\ge 3$. Then
$\chi_{3-vsdt}(T)\leq 2\Delta(T)+1$.
\end{cor}

\section{Preliminaries}
In the section, we give some useful lemmas which will play an important role in the proof of our main results.

\begin{lem}\label{lem-1}   %考虑一下引理改成这样
Let $A,B\subseteq S$ and $A\not=B$. Then there exists at most one $c\in S$ such that $A= B\cup \{c\}$ or $B= A\cup \{c\}$.
\end{lem}
\begin{proof}
Without loss of generality, we assume that $|B|\le |A|$, then one can obtain that either $A\not= B\cup \{x\}$ for any $x\in S$, or there is some $c\in S$ such that $A=B\cup \{c\}$. Now we prove that $c$ is unique if the latter one holds. Assume, to the contrary, that there are two distinct $c,c'\in S$ such that the letter one holds. Since $A= B\cup \{c'\}$, we have $c'\in A$ and $c'\not\in B$. Similarly, we have $c\in A$ and $c\not\in B$. This means that $A$ is more than $B$ with at least $2$ elements $c'$ and $c$, it is a contradiction.
\end{proof}

\begin{lem}\label{lem-2}
Suppose $A,B\subseteq S$ and $A\not=B$. Then the element $c$ yielding $A\cup \{c\}= B\cup \{c\}$, has at most one in $S$.
% Then there exists at most one $c\in S$ such that $A\cup \{c\}= B\cup \{c\}$.
\end{lem}
\begin{proof}
Since $A,B\subseteq S$ and $A\not=B$, the equation of $A\cup \{c\}= B\cup \{c\}$ implies that either $A\subset B$ and $|B|-|A|=1$, or $B\subset A$ and $|A|-|B|=1$. Furthermore, it follows that there at most exists one $c\in B$ but $c\not\in A$, or $c\in A$ but $c\not\in B$. So the proof completes.
\end{proof}

\begin{lem}[\!\cite{zhang-3}]\label{lem-3}
Let $P_{n}$ be a path of order $n(\geq3)$. If $n$ is odd, then $\chi_{1-vsdt}(P_{n})=4$, otherwise, $\chi_{1-vsdt}(P_{n})=5$.
\end{lem}
\begin{lem}[\!\cite{zhang-3}]\label{lem-4}
Let $C_{n}$ be a cycle of order $n(\geq3)$. Then $\chi_{1-vsdt}(C_{n})\le 5$.
\end{lem}

%\begin{lem}\label{lem-8}
%Suppose $A,B\subseteq S$ and $A=B$. If $\exists x \in S, A\cup\{x\}\not=B$. Then $x\not\in B$.
%\end{lem}

\begin{lem}[\!\cite{limc}]\label{lem-5}
For any tree $T$ with at least three vertices, $\chi'_{s}(T,2)\le \Delta(T)+1$. Moreover, if for any two vertices $u$ and $v$ with maximum degree, and $d(u,v)\ge3$, then $\chi'_{s}(T,2)=\Delta(T)$.
\end{lem}

\begin{lem}[\!\cite{limc}]\label{lem-6}
For any tree $T$ with $\Delta(T)\ge 3$, $\chi'_{s}(T,3)\le 2\Delta(T)-1$.
\end{lem}

From Definition \ref{def-1} we see that the following lemma always holds.
\begin{lem}\label{lem-7}
Let $G$ be a graph. Then $\chi_{r-vsdt}(G)\ge \Delta(G)+1$, and further $\chi_{r-vsdt}(G)\ge \Delta(G)+2$ if $G$ has two vertices $u,v$ with maximum degree and $d(u,v)\le r$.
\end{lem}

\section{Proofs of main results}
In this section, we first give a proof of Theorem \ref{thm-3} in the following.
\begin{proof0}
 \rm For convenience, we denote by $\chi'_{s}(G,r)=p$ and $\chi(G)=q (q\ge p)$. Let $f'$ be a $p$-$D(r)$-SEC of $G$ with colors $\{1,2,\ldots, p\}$, and $f^{*}$ a proper $q$-vertex coloring of $G$ with colors $\{p+1,p+2,\ldots, p+q\}$. We here define a new proper $(p+q)$-total coloring $f$ as follows
\begin{equation*}
f(z)=\begin{cases}
f'(z), & \forall z\in E(G)\\
f^{*}(z),  &\forall z\in V(G)\\
\end{cases}
\end{equation*}
Clearly, $f$ is a proper $(p+q)$-total coloring of $G$, and $\forall y\in V(G)$, $C_{f}\langle y\rangle=C_{f'}(y)\cup \{f^{*}(x)|x\in N[y]\}$ where $C_{f'}(y)\subseteq \{1,2,$ $\ldots, p\}$ and $\{f^{*}(x)|x\in N[y]\}\subseteq \{p+1,p+2,\ldots, p+q\}$.

Let $u,v$ be any two distinct vertices of $G$. %{\color{red} with $d(u,v)\le 2$}    不需要距离限制
Since $f'$ is a $p$-$D(r)$-SEC of $G$, it follows from Definition \ref{def-2} that $C_{f}(u)\not= C_{f}(v)$. Furthermore, $C_{f'}(u)\cup \{f^{*}(x)|x\in N[u]\}\not= C_{f'}(v)\cup \{f^{*}(x)|x\in N[v]\}$ because $\{1,2,\ldots, p\}\cap \{p+1,p+2,\ldots, p+q\}=\emptyset$. Namely, $C_{f}\langle u\rangle \not =C_{f}\langle v\rangle$. Thus, $f$ is a $(p+q)$-$D(r)$-VSDTC of $G$.

The proof completes.  \hfill$\square$
\end{proof0}

Next, we give some notations before proving Theorem \ref{thm-4}. For a graph $G$, let $A$ be a proper subset of $V(G)\cup E(G)$, and $f$ an \emph{$r$-vertex-strongly-distinguished} mapping from $A$ to $\{1,2,\ldots,\kappa\}$. Then a vertex $u$ of $A$ is called to be `\emph{good}' if all elements of $N\langle u\rangle$ is colored under the current $f$. Otherwise, the vertex of $u$ is called to be `\emph{bad}'. In addition, two vertices $u$ and $v$ are called \emph{strongly distinguished} under $f$ if the two vertices satisfy $C_{f}\langle u\rangle \not=C_{f}\langle v\rangle$.

\begin{proof1}
\rm Let $G$ be a $k$-degenerated graph and $\lambda=k\Delta+3$, where $k\geq 2$. Clearly, $\Delta\geq2$ since $G$ has no isolated edges as its components.
When $\Delta =2$, by Lemmas \ref{lem-3} and \ref{lem-4} we know that the conclusion holds. %有可能是1-degenerated graph，我觉得还是用原来的证明过程。
When $\Delta \ge 3$, it implies that $|V(G)|\ge 4$, we will prove the results by induction on $|V(G)|$.

If $|V(G)|=4$, the theorem is obvious; suppose otherwise, that $H$ is a vertex-induced subgraph of $G$ with $|V(H)|$ less than $|V(G)|$, and $H$ has a $\lambda$-$D(1)$-VDSTC $f$. We here verify that $G$ also admits a $\lambda$-$D(1)$-VDSTC.

Let $v$ be a vertex of $G$ with $d(v)=\delta$, and $N(v)=\{v_1,v_2,\ldots,v_\delta\}$. We denote by $e_i=vv_i$ and $N(v_i)=\{w_{i1},w_{i2},\ldots, w_{it_i}, v\}$, where $i=1,2,\ldots, \delta$. As is shown in Fig.\ref{figure-1}.

\vspace{-3.0cm}
\begin{figure}[!hbt]
    \centering
    \includegraphics[height=20 cm, width =15 cm]{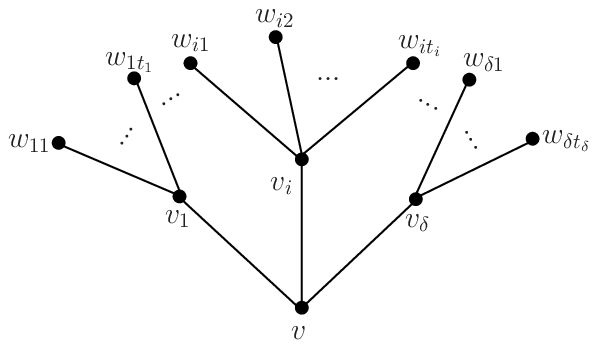}
\vspace{-13.5cm}\caption{\footnotesize The illustration of $G$}\label{figure-1}
\end{figure}

Set $H:=G\backslash \{v\}$. By the induction hypothesis above, suppose $H$ has a $\lambda$-$D(1)$-VDSTC $f$ such that for any $uw\in E(H)$, $C_{f}\langle u\rangle\not =C_{f}\langle w\rangle$. In order to extend $f$ to a $\lambda$-$D(1)$-VDSTC $f'$ of $G$, one should firstly consider the coloring of $v$. Now, a claim is given in the following.

\begin{claim}\label{claim-1}
There is some color $c_{0}\in \{1,2,\ldots,\lambda\}$ such that
\begin{equation}\label{eq-1}
c_0\not\in \{f(v_i)\mid 1\leq i \leq\delta\} \text{~~and~~} C_{f}\langle w_{ij}\rangle\not=C_{f}\langle v_{i}\rangle\cup\{c_{0}\}
\end{equation}
for any $1\le j\le t_{i}$ and $1\le i\le \delta$.
\end{claim}
\begin{proof}
For any neighbor $w_{ij}$ of each $v_{i}$ $(1\le j\le t_{i},1\le i\le \delta)$, we have $C_{f}\langle w_{ij}\rangle\not=C_{f}\langle v_{i}\rangle$ since $f$ is a $D(1)$-VDSTC of $H$. Furthermore, it follows from Lemma \ref{lem-1} that there is at most one $c_{ij} \in \{1,2,\ldots,\lambda\}$ satisfying $C_{f}\langle w_{ij}\rangle =C_{f}\langle v_{i}\rangle\cup\{c_{ij}\}$. Note that $t_{i}\le \Delta-1$ and $|\{f(v_i): 1\leq i \leq\delta\}|\leq \delta$. Thus, there are at most $\delta+\delta(\Delta-1)$ colors $c$ such that $C_{f}\langle w_{ij}\rangle =C_{f}\langle v_{i}\rangle\cup\{c\}$.

Since $|\{f(v_i): 1\leq i \leq\delta\}|\le \delta\leq k$, and $\lambda-(\delta+\delta(\Delta-1))=k\Delta+3-\delta\Delta=(k-\delta)\Delta+3> 1$, and so one can conclude that $c_{0}$ exists.
\end{proof}

Based on Claim \ref{claim-1}, we define a new coloring $f_{0}$ of $H\cup \{v\}$ as follows
\begin{equation*}
f_{0}(z)=\begin{cases}
c_{0}, &\hspace{1cm}z=v\\
f(z),  &\forall z \in V(H)\cup E(H)\\
\end{cases}
\end{equation*}

Next we will successively color the edges $e_{1},e_{2},\ldots, e_{\delta}$ of $v$ and yield that $v_{1},v_{2},\ldots, v_{\delta}$ become `good' vertices. For convenience, we denote $f_{i}$ by the current coloring of $G$ after the $i$-th edge $e_{i}$ of $v$ is colored, and $R_i$ the number of colors on the corresponding $e_i$ such that the current $f_{i}$ is not a $D(1)$-VDSTC of $G$, and also call it to be the \emph{forbidden colors number}, where $1\leq i\leq\delta$. Especially, those colors are called the \emph{forbidden colors}. In addition, let $C_{f_{i}}\langle z\rangle$ denote the color-set of a vertex $z$ consisting of the colors assigned on $N[z]$ and the incident edges of $z$ under the current  $f_{i}$.

For the positive integer of $\delta$, we distinct three cases in the following.

\textbf{Case A.} If $\delta\ge 3$, then we firstly consider the coloring of $e_1$ to obtain $f_{1}$. Obviously, for the coloring of $f_{1}$, there are two types forbidden colors on $e_{1}$ below.
\begin{enumerate}
\item [(a)] To ensure $f_{1}$ is a local proper coloring, the forbidden colors include the following: $f_{0}(v),f_{0}(v_{1})(=f(v_{1}))$ and $f_{0}(v_{1}w_{1j})(=f(v_{1}w_{1j}))$ for $j=1,2,\ldots,t_{1}$;
\item [(b)] To ensure $v_{1}$ and $w_{1j}$ are strongly distinguished, we notice that $C_{f_{0}}\langle w_{1j}\rangle\not =C_{f_{0}}\langle v_{1}\rangle$ for $j=1,2,\ldots,t_{1}$. Then by Lemma \ref{lem-1}, there is at most one color (say) $c_{1j}$ to yield $C_{f_{0}}\langle w_{1j}\rangle=C_{f_{0}}\langle v_{1}\rangle\cup\{c_{1j}\}$ for each $j$, where $j=1,2,\ldots,t_{1}$.
\end{enumerate}
From the above, it has at most $t_{1}+2$ forbidden colors in the subcase (a), and at most $t_{1}$ forbidden colors in the subcase (b). Thus, $R_{1}\le 2+2t_{1}\le 2\Delta$. Clearly, $\lambda-2\Delta>0$. So one can choose one color from $\{1,2,\ldots,\lambda\}$ except for the above $2\Delta$ forbidden colors to dye $e_{1}$, and the reminders remain unchange. At this point, we write the revised coloring function as $f_{1}$. Then under the coloring $f_{1}$, the vertex $v_{1}$ becomes a `good' vertex, and other vertices $v$ and $v_{i}(i=2,3,\ldots, \delta)$ are also `bad'.

%Thus, $f_{1}$ is clearly a $\lambda$-$D(1)$-VDSTC of $G$.

For the coloring of $f_{i}$, we successively promote the following procedure on $e_{i}$, where $i=2,3,\ldots,\delta-2$.
\begin{enumerate}
\item [(a')] To ensure $f_{i}$ is a local proper coloring, the forbidden colors include the following: $f_{i-1}(v)(=f_{0}(v))$, $f_{i-1}(v_{i})(=f(v_{i}))$, $f_{i-1}(e_{1}), \ldots, f_{i-1}(e_{i-1})$ and $f_{i-1}(v_{i}w_{ij})$ $(=f(v_{i}w_{ij}))$ for $j=1,2,\ldots,t_{i}$;
\item [(b')] Ensuring $v_{i}$ and $w_{ij}$ are strongly distinguished, we notice that $C_{f_{i-1}}\langle w_{ij}\rangle\not=C_{f_{i-1}}\langle v_{i}\rangle$ for $j=1,2,\ldots,t_{i}$. Then by Lemma \ref{lem-1}, there is at most one color (say) $c_{ij}$ to yield $C_{f_{i-1}}\langle w_{ij}\rangle=C_{f_{i-1}}\langle v_{i}\rangle\cup\{c_{ij}\}$ for each $j$, where $j=1,2,\ldots,t_{i}$.
\end{enumerate}
According to the above, two subcases have at most $t_{i}+i+1$ and $t_{i}$ forbidden colors, respectively, so we have $R_{i}\le 2t_{i}+i+1\le 2\Delta+i-1$. Clearly, $\lambda-R_{i}\ge \lambda-(2\Delta+i-1)=(k-2)\Delta+4-i\ge (\delta-1.5)^{2}+3.75>1$ for $2\le i\le \delta-2$ (Note that in this case $k\geq \delta\geq i+2\geq 4$). Thus, for each $i$, one can orderly choose one color from $\{1,2,\ldots,\lambda\}$ except for the $2\Delta+i-1$ forbidden colors to dye $e_{i}$, and the reminders remain unchange. We here write the resulting coloring function as $f_{i}$ for $i=2,3,\ldots,\delta-2$. Under each $f_{i}$ by order, $v_{1},\ldots,v_{i-1}$ and $v_{i}$ become `good' vertices, and $v_{i+1},\ldots, v_{\delta}$ and $v$ are also `bad' vertices.

For the coloring $f_{\delta-1}$, there are three types of forbidden colors on $e_{\delta-1}$ to be considered below.
\begin{enumerate}
\item [(i)] To ensure $f_{\delta-1}$ is a local proper coloring, the forbidden colors include the following: $f_{\delta-2}(v)(=f_{0}(v)),f_{\delta-2}(v_{\delta-1})(=f(v_{\delta-1}))$, $f_{\delta-2}(e_{1}), \ldots, f_{\delta-2}(e_{\delta-2})$ and $f_{\delta-2}(v_{\delta-1}$ $w_{\delta-1,j})(=f(v_{\delta-1}w_{\delta-1,j}))$ for $j=1,2,\ldots,t_{\delta-1}$;
\item [(ii)] Ensuring $v_{\delta-1}$ and $w_{\delta-1,j}$ are strongly distinguished.  Note that $C_{f_{\delta-2}}\langle w_{\delta-1,j}\rangle=C_{f_{\delta-2}}\langle v_{\delta-1}\rangle$. Then by Lemma \ref{lem-1}, there is at most one color (say) $c_{\delta-1,j}$ to yield $C_{f_{\delta-2}}\langle w_{\delta-1,j}\rangle=C_{f_{\delta-2}}\langle v_{\delta-1}\rangle\cup\{c_{\delta-1,j}\}$ for each $j$, where $j=1,2,\ldots,t_{\delta-1}$;
\item [(iii)] There are at most $\delta$ forbidden colors of $f_{\delta-1}$ under the following restrictions: $C_{f_{\delta-1}}\langle v\rangle\not =C_{f_{\delta-1}}\langle v_{\delta}\rangle$ if $C_{f_{\delta-2}}\langle v\rangle=C_{f_{\delta-2}}\langle v_{\delta}\rangle$, otherwise $|C_{f_{\delta-1}}\langle v\rangle\oplus C_{f_{\delta-1}}\langle v_{\delta}\rangle|\ge 2$ where $\oplus$ is the symmetric minus.
\end{enumerate}
\begin{claim}\label{claim-2}
Under the above assumptions of \emph{(iii)}, if $f_{\delta-1}$ is a local proper coloring, and $C_{f_{\delta-1}}\langle w_{\delta-1,j}\rangle$ $\not=C_{f_{\delta-1}}\langle v_{\delta-1}\rangle$ for $j=1,2,\ldots,$ $t_{\delta-1}$. There are at most $\delta$ forbidden colors for the following restrictions:
\begin{enumerate}
 \item [$\bullet$] if $C_{f_{\delta-2}}\langle v\rangle=C_{f_{\delta-2}}\langle v_{\delta}\rangle$, then $C_{f_{\delta-1}}\langle v\rangle\not =C_{f_{\delta-1}}\langle v_{\delta}\rangle$, otherwise,
  \item [$\bullet$] $|C_{f_{\delta-1}}\langle v\rangle\oplus C_{f_{\delta-1}}\langle v_{\delta}\rangle|\ge 2$.
\end{enumerate}
\end{claim}
\begin{proof}
We here distinct two cases to discuss the number of forbidden colors for the above two restrictions.

\textbf{Case 1.} $C_{f_{\delta-1}}\langle v\rangle\not =C_{f_{\delta-1}}\langle v_{\delta}\rangle$ if $C_{f_{\delta-2}}\langle v\rangle=C_{f_{\delta-2}}\langle v_{\delta}\rangle$.

From the definition of $f_{i}$ we know that $C_{f_{\delta-2}}\langle v\rangle=\{f_{0}(v),f(v_{\delta}),f(v_{\delta-1}), f(v_{i}),f_{\delta-2}(e_{i})|$  $i=1,2,\ldots, \delta-2\}$ and $C_{f_{\delta-2}}\langle v_{\delta}\rangle=C_{f}\langle v_{\delta}\rangle\cup \{f_{0}(v)\}$, but $C_{f_{\delta-1}}\langle v\rangle = C_{f_{\delta-2}}\langle v\rangle \cup \{f_{\delta-1}(e_{\delta-1})\}$ and $C_{f_{\delta-1}}\langle v_{\delta}\rangle=C_{f_{\delta-2}}\langle v_{\delta}\rangle$.
Thus, $C_{f_{\delta-1}}\langle v\rangle =C_{f_{\delta-1}}\langle v_{\delta}\rangle$ implies that $f_{\delta-1}$ should be taken one color from $\{f(v_{\delta}), f(v_{i})|$  $i=1,2,\ldots, \delta-2\}$. Therefore, there are at most $\delta-1$ forbidden colors.

\textbf{Case 2.} $|C_{f_{\delta-1}}\langle v\rangle\oplus C_{f_{\delta-1}}\langle v_{\delta}\rangle|\ge 2$. By the analysis of $|C_{f_{\delta-2}}\langle v\rangle\oplus C_{f_{\delta-2}}\langle v_{\delta}\rangle|$ to study the forbidden colors, there are three subcases to be considered in the following.

\textbf{Case 2.1.} $|C_{f_{\delta-2}}\langle v\rangle\oplus C_{f_{\delta-2}}\langle v_{\delta}\rangle|=1$. It is easy to find that the number of forbidden colors is maximum if $C_{f_{\delta-2}}\langle v\rangle\subset C_{f_{\delta-2}}\langle v_{\delta}\rangle$ and $|C_{f_{\delta-2}}\langle v_{\delta}\rangle|-|C_{f_{\delta-2}}\langle v\rangle|=1$. In this case, $C_{f_{\delta-2}}\langle v_{\delta}\rangle$ has just one color different from that of $C_{f_{\delta-2}}\langle v\rangle$, and the other colors is the same as $C_{f_{\delta-2}}\langle v\rangle$. From the proof of Case1. we know that $f_{\delta-1}$ has at most $\delta-1$ forbidden colors, plus the one color different from that of $C_{f_{\delta-2}}\langle v\rangle$. Consequently, there are at most $\delta$ forbidden colors.

\textbf{Case 2.2.} $|C_{f_{\delta-2}}\langle v\rangle\oplus C_{f_{\delta-2}}\langle v_{\delta}\rangle|=2$. Then the forbidden colors should be only taken form $C_{f_{\delta-2}}\langle v\rangle\oplus C_{f_{\delta-2}}\langle v_{\delta}\rangle$, thus, there are $2$ forbidden colors.

\textbf{Case 2.3.} $|C_{f_{\delta-2}}\langle v\rangle\oplus C_{f_{\delta-2}}\langle v_{\delta}\rangle|\ge 3$. Clearly, taking one color from $\{1,2,\ldots,\lambda\}$ to dye $e_{\delta-1}$ yields $|C_{f_{\delta-2}}\langle v\rangle\oplus C_{f_{\delta-2}}\langle v_{\delta}\rangle|$ at most descend $1$. Hence, it has no forbidden colors in the case.

From the above discussion, there are at most $\delta$ forbidden colors for the statements.
\end{proof}

In light of the above, there is at most $\delta+\Delta-1$ forbidden colors in (i), and $\Delta-1$ forbidden colors in (ii), and $\delta$ forbidden colors in (iii). Therefore, $R_{\delta-1}\le (\delta+\Delta-1)+ (\Delta-1)+\delta=2\Delta+2\delta-2$. Notice that $\lambda-R_{\delta-1}=k\Delta+3-(2\Delta+2\delta-2)\ge(\delta-2)\Delta-2\delta+5
\ge(\delta-2)\delta-2\delta+5=(\delta-2)^2+1> 1$. So one can also choose one color from $\{1,2,\ldots,\lambda\}$ except for the $2\Delta+2\delta-2$ forbidden colors to dye $e_{i}$, and similarly obtain the new coloring $f_{\delta-1}$. Under the coloring $f_{\delta-1}$, it is not difficult to find that $v$ and $v_{\delta}$ are also `bad' vertices.

%At last step, we first give an obvious claim in the following.
%\begin{claim}\label{claim-3}
%Suppose $S_{1},S_{2}\subset S$ and $S_{1}=S_{2}$. If ~$\exists x\in S$ and $S_{1}\cup\{x\}\not=S_{2}$, then $x\not \in S_{2}$. Moreover, $\forall y \in S$ and $y\not= x$. Then $S_{1}\cup\{x,y\}\not=S_{2} \cup\{y\}$.
%\end{claim}

 At last, we consider the coloring $f_{\delta}$ on $e_{\delta}$. According to the analysis of the current coloring, there are four types forbidden colors in the following.
\begin{enumerate}
  \item [(i')] To ensure $f_{\delta}$ is a proper total coloring of $G$, the forbidden colors includes the following: $f_{\delta-1}(=f_{0}(v)),f_{\delta-1}(v_{\delta})(=f(v_{\delta}))$, $f_{\delta-1}(e_{i})$ for $i=1,2,\ldots,\delta-1$ and $f_{\delta-1}(v_{\delta}w_{\delta j})$ $(=f(v_{\delta}w_{\delta j}))$ for $j=1,2,\ldots,t_{\delta}$;
  \item [(ii')] To ensure $v_{\delta}$ and $w_{\delta j}$ are strongly distinguished, it follows from Lemma \ref{lem-1} that there is at most one color $c_{\delta j}$ to yield $C_{f_{\delta-1}}\langle w_{\delta j}\rangle=C_{f_{\delta-1}}\langle v_{\delta}\rangle\cup\{c_{\delta j}\}$ for each $j$, where $j=1,2,\ldots,t_{\delta}$;
  \item [(iii')] To ensure $v$ and $v_{i}(i=1,2,\ldots,\delta-1)$ are strongly distinguished, there will present two possibilities:
      \begin{enumerate}
        \item [(1)] when $C_{f_{\delta-1}}\langle v\rangle = C_{f_{\delta-1}}\langle v_{i}\rangle$ for some $i\in\{1,2,\ldots, \delta-1\}$, the forbidden color of $e_{\delta}$ should be taken from one of $\{f_{\delta-1}(v_{j})|j=1,2,\ldots,\delta-1\}$. Except for those forbidden colors and that of (i'), one can take any one color $c$ such that $C_{f_{\delta-1}}\langle v\rangle\cup \{c\}\not= C_{f_{\delta-1}}\langle v_{i}\rangle$.
      \item [(2)] when $C_{f_{\delta-1}}\langle v\rangle\not=C_{f_{\delta-1}}\langle v_{i}\rangle $for some $i\in\{1,2,\ldots, \delta-1\}$, then by Lemma \ref{lem-1},  $e_{\delta}$ has at most one forbidden color $c^{i}$ (say).
      \end{enumerate}
      \quad \quad If $(1)$ appears $0$ time, and $(2)$ appears $\delta-1$ times, in order to ensure the strongly distinguished property of $v$ and $v_{1},v_{2}, \ldots, v_{\delta-1}$, then the number of the forbidden colors of $e_{\delta}$ is at most $\delta-1$, say $c^{1},c^{2},\ldots, c^{\delta-1}$; if $(1)$ appears $t~(1\le t <\delta-1)$ times, and $(2)$ appears $\delta-1-t$ times, in order to ensure the strongly distinguished property of $v$ and $v_{i}(i=1,2,\ldots,\delta-1)$, then the number of the forbidden colors of $e_{\delta}$ is at most $\delta-1+\delta-1-t\le 2\delta-3$; if $(1)$ appears $\delta-1$ times, and $(2)$ appears $0$ time, then the number of the forbidden colors of $e_{\delta}$ is at most $\delta-1$. To sum up above, $e_{\delta}$ has at most $\max\{\delta-1,2\delta-3\}=2\delta-3$ forbidden colors in (iii').

\item [(iv')] To ensure $v$ and $v_{\delta}$ are strongly distinguished. According to (iii), there are two cases to discuss.

      \textbf{Case 1.} If $C_{f_{\delta-1}}\langle v\rangle\not =C_{f_{\delta-1}}\langle v_{\delta}\rangle$ while
      $C_{f_{\delta-2}}\langle v\rangle=C_{f_{\delta-2}}\langle v_{\delta}\rangle$. It can be checked that $e_{\delta-1}$ has no
      forbidden color as $f_{\delta}$ is a proper total coloring of $G$.

      %%From Claim \ref{claim-3},

      \textbf{Case 2.} If $|C_{f_{\delta-1}}\langle v\rangle\oplus C_{f_{\delta-1}}\langle v_{\delta}\rangle|\ge 2$, then the coloring of $e_{\delta}$ at most yields the symmetric minus decline $1$, so the number of forbidden color is also zero.
\end{enumerate}

From the above we known that there are at most $\delta+\Delta$ forbidden colors in (i'), and $\Delta-1$ forbidden colors in (ii'), and $2\delta-3$ forbidden colors in (iii'), respectively. Note that (iv') has no forbidden color. So we have
$$R_{\delta}\le (\delta+\Delta) +(\Delta-1)+ 2\delta-3 =2\Delta+3\delta-4.$$

Hence, $\lambda-R_{\delta}\ge k\Delta+3-(2\Delta+3\delta-4)\ge (\delta-2)\Delta-3\delta+7\ge (\delta-2)\delta-3\delta+7=(\delta-2.5)^2+0.75\geq 1$ since $\delta$ is an integer. Thus, $f_{\delta}$ is a $\lambda$-$D(1)$-VDSTC of $G$.

\textbf{Case B.} If $\delta=2$, then one should consider three types of the forbidden colors for the coloring $f_{1}$ on $e_{1}$: subcases (a), (b) and (c) respectively. The first two subcases see \textbf{Case A}, the repetitions will not be made here. Now we only give subcase (c) as follows:
\begin{enumerate}
%\item [(a)] To ensure $f_{1}$ is a local proper coloring, the forbidden colors include the following: $f_{0}(v), f_{0}(v_{1})(=f(v_{1}))$ and $f_{0}(v_{1}w_{1j})(=f(v_{1}w_{1j}))$ for $j=1,2,\ldots,t_{1}$;
%\item [(b)] Ensuring $v_{1}$ and $w_{1,j}$ are strongly distinguished.  Note that $C_{f_{0}}\langle w_{1j}\rangle\not=C_{f_{0}}\langle v_{1}\rangle$. Then by Lemma \ref{lem-1}, there is at most one color (say) $c_{1j}$ to yield $C_{f_{0}}\langle w_{1j}\rangle=C_{f_{0}}\langle v_{1}\rangle\cup\{c_{1j}\}$ for each $j$, where $j=1,2,\ldots,t_{1}$;
\item [(c)] To avoid $C_{f_{1}}\langle v\rangle =C_{f_{1}}\langle v_{2}\rangle$, there is at most one forbidden color for $e_{1}$  since if either $C_{f_{0}}\langle v\rangle\not =C_{f_{0}}\langle v_{2}\rangle$, by Lemma \ref{lem-1} there is a possible color for $e_{1}$ to yield $C_{f_{1}}\langle v\rangle =C_{f_{1}}\langle v_{2}\rangle$, or $C_{f_{0}}\langle v\rangle =C_{f_{0}}\langle v_{2}\rangle$ then, the forbidden color of $e_{1}$ is only taken by $f(v_{2})$.

% \item [(c)] \textcolor[rgb]{0.00,0.00,1.00}{To avoid $v$ and $v_{2}$ are not strongly distinguishing when $e_{2}$ is colored, there are at most two forbidden colors since: if $C_{f_{1}}\langle v\rangle\not =C_{f_{1}}\langle v_{2}\rangle$ then by Lemma \ref{lem-2}, there is a possible color such that $C_{f_{2}}\langle v\rangle =C_{f_{2}}\langle v_{2}\rangle$; otherwise, there are two possible colors whence $C_{f_{1}}\langle v\rangle=\{f_{1}(v),f_{1}(v_{1}),f_{1}(v_{2}),f_{1}(e_{1})\}=C_{f_{1}}\langle v_{2}\rangle$, it implies that $v_{2}$ has other neighbors that are only colored by one of $f_{1}(v_{1})$ and $f_{1}(e_{1})$, on contrary, $f_{1}(e_{1})$ should be taken from one of the two colors that present on other incident elements except for $v$.}

\end{enumerate}
From \textbf{Case A.} we know that there are at most $2\Delta$ forbidden colors in (a) and (b), and there is one forbidden color in (c). Therefore, $R_{1}\le (\Delta+1)+ (\Delta-1)+1=2\Delta+1$. Note that $\lambda-R_{1}=k\Delta+3-(2\Delta+1)\ge(k-2)\Delta+2\ge 2$. So one can also choose one color from $\{1,2,\ldots,\lambda\}$ except for the above $2\Delta+1$ forbidden colors to dye $e_{1}$, and obtain the new coloring $f_{1}$. Under the coloring $f_{1}$,  it can be seen that $v$ and $v_{2}$ are also `bad' vertices.

For the coloring $f_{2}$ on $e_{2}$, we also consider three types of the forbidden colors: subcases (a'), (b'), (c') and (d') respectively. Subcases (a') and (b') shown in \textbf{Case A} while $i=2$, here is no longer repeated it. Now we give subcases (c') and (d') in the following.
\begin{enumerate}
%\item [(a')] To ensure $f_{2}$ is a local proper coloring, the forbidden colors include the following: $f_{1}(v)(=f_{0}(v)), f_{1}(v_{2})(=f(v_{2})),f_{1}(e_{1})$ and  $f_{1}(v_{2}w_{2j})(=f(v_{2}w_{2j}))$ for $j=1,2,\ldots,t_{2}$;
%\item [(b')] Ensuring $v_{2}$ and $w_{2j}$ are strongly distinguished. Note that $C_{f_{1}}\langle v_{2}\rangle \not=C_{f_{1}}\langle w_{2j}\rangle$. Then by Lemma \ref{lem-1}, there is at most one color (say) $c_{2j}$ to yield $C_{f_{1}}\langle w_{2j}\rangle=C_{f_{1}}\langle v_{2}\rangle\cup\{c_{2j}\}$ for each $j$, where $j=1,2,\ldots,t_{1}$;
\item [(c')] To obtain $C_{f_{2}}\langle v\rangle \not=C_{f_{2}}\langle v_{2}\rangle$, $e_{2}$ has no any forbidden color since $|C_{f_{1}}\langle v\rangle \oplus C_{f_{1}}\langle v_{2}\rangle|\ge 2$, which can be easily deduced from the above subcase (c) and Claim \ref{claim-2}, here we omit its proof.
\item [(d')] To ensure that $C_{f_{2}}\langle v\rangle\not \not=C_{f_{2}}\langle v_{1}\rangle$, there is at most one forbidden color for $e_{2}$ whether $C_{f_{1}}\langle v\rangle$ and $C_{f_{1}}\langle v_{1}\rangle$ are identical or not.
\end{enumerate}
Analogous to the above of $R_{1}$, one can get $R_{2}\le (\Delta+2)+ (\Delta-1)+0+1=2\Delta+2$. Thus, there is at least one color from $\{1,2,\ldots,\lambda\}$ except for the above $2\Delta+2$ forbidden colors to dye $e_{2}$. Consequently, $f_{2}$ is a $\lambda$-$D(1)$-VDSTC of $G$.

\textbf{Case C.} If $\delta=1$, then $f_{1}$ has two types of the forbidden colors on $e_{1}$: subcases (a) and (b) respectively, see \textbf{Case A}. Clearly, there are at most $2\Delta$ forbidden colors in those situations. On the other hand, one should consider the strongly distinguished property of $v$ and $v_{1}$. Clearly, $C_{f_{0}}\langle v\rangle \not =C_{f_{0}}\langle v_{1}\rangle$ since if not, $|C_{f_{0}}\langle v\rangle|=2$ but $|C_{f_{0}}\langle v_{1}\rangle|\ge 3$, a contradiction. So it follows from Lemma \ref{lem-1} that there is at most one forbidden color here. Therefore, $R_{1}\le 2\Delta+1$, and so, it can be deduced that $f_{1}$ is a $\lambda$-$D(1)$-VDSTC of $G$.

Sum up the above, the proof is completed.  \hfill$\square$
\end{proof1}

\begin{proof2}
\rm Since $T$ is a bipartite graph, $T$ admits a $2$-vertex coloring, namely $\chi(T)=2$. From Lemma \ref{lem-5}, one can obtain $\chi'_{s}(T,2)\le \Delta(T)+1$. Thus, it follows from Theorem \ref{thm-3} that $\chi_{2-vsdt}(T)\le \Delta(T)+3$.

Similarly, if $d(u,v)\ge3$ for any two vertices $u,v\in V(T)$ with maximum degree, one can also obtain that $\chi_{2-vsdt}(T)\le \Delta(T)+2$.

The proof follows. \hfill$\square$
\end{proof2}

\section{Further works and problems }

As far as we know, for the \emph{(adjacent-) vertex-distinguishing colorings}, there exist some proper subgraphs $H$ of $G$ such that its \emph{(adjacent-) vertex-distinguishing
chromatic number} is greater than that of $G$, see \cite{zhang-4,zhang-1,wen} for instance.
Thus, we propose the following problem as a further work.  % 把你ARS上的文章引用上
\begin{prob}\label{problem-1}
For any positive integer $r$, if $H$ is a proper subgraph of $G$, when $\chi_{r-vsdt}( H)\le \chi_{r-vsdt}(G)$ is always true ?
\end{prob}
From Theorem 6 of literature \cite{zhang-3}, one can obtain  $\chi_{1-vsdt}(K_{n})\ge n+\lceil\log_{2}^{n}\rceil$, so too is $\chi_{r-vsdt}(K_{n})$ for $r\ge 2$ since $\mathbf{d}(K_{n})=1$. In fact, if $G$ admits an $r$-VDSTC, then the restriction of the color-set on each vertex of $G$, that is, any two vertices of $G$ with distance at most $r$ are \emph{vertex-strongly-distinguishing}, can be viewed as two adjacent vertices with \emph{$1$-vertex-strongly-distinguishing property}. Moreover, Zhang et al. in \cite{zhang-3} also obtained the accurate value of  $\chi_{1-vsdt}(K_{n})$ for $3\le n\le 18$, so we can get the following table.

\begin{table}[htbp]
 \caption{$\chi_{r-vsdt}(K_{n})$ for $3\le n\le 18$}\label{tab-1}
 \begin{tabular}{c|cccccccccccccccccc}
  \toprule
$n$ &3 & 4 &5&6&7&8&9&10&11&12&13&14&15&16&17&18  \\
  \midrule
$\chi_{r-vsdt}(K_{n})$ &5 & 6 &8&10&10&11&13&14&15&16&17&19&19&20&22&23 \\
  \bottomrule
 \end{tabular}
\end{table}
Hence, combining the above we have the following conjecture.
\begin{conj}\label{conj-1}
Let $G$ be a graph on $n$ vertices and without isolated edges. Then
$$\chi_{r-vsdt}(G)\le n+\lceil\log_{2}n\rceil+1$$
where the equality holds if $n=2^{k}-2$.
\end{conj}

From Hatami's results\cite{hatami} and Theorem \ref{thm-3} we know that if $\Delta(G)$ is very large, then $\chi_{1-vsdt}(G)\le 2\Delta(G)+300$. So we give the following conjecture.
\begin{conj}\label{conj-2}
For any graph $G$ without no isolated edges, there is a positive constant $c$ such that
$$\chi_{1-vsdt}(G)\le 2\Delta(G)+c.$$
\end{conj}

%\end{CJK*}

\begin{thebibliography}{20}
{\footnotesize
\bibitem{bondy} J.A. Bondy, U.S.R. Murty, \emph{Graph Theory}, Springer, London, 2008. \vspace{-0.25cm}

\bibitem{Burris} A.C. Burris, R.H. Schelp, \emph{Vertex distinguishing proper edge-coloring}, J. Graph Theory, \textbf{26}(1997), no. 2, 70--82.   \vspace{-0.25cm}

\bibitem{zhang-4} Z.F. Zhang, L.Z. Liu, J.F. Wang, \emph{adjacent strong Edge coloring of graphs}, Appl. Math. Lett., \textbf{15}(2002), no. 3, 623--626. \vspace{-0.25cm}

\bibitem{hatami} H. Hatami, \emph{$\Delta+300$ is a bound on the adjacent vertex distinguishing edge chromatic number}, J. Comb. Theory Ser. B.,  \textbf{95}(2005), no. 2, 246--256.   \vspace{-0.25cm}

\bibitem{balister} P.N. Balister, J. Lehel, R.H. Schelp, \emph{Adjacent vertex distinguishing edge colorings}, Siam J. on Disc. Math., \textbf{21}(2007), no. 1, 237--250.   \vspace{-0.25cm}

\bibitem{limc} S. Akbaria, H. Bidkhori, N. Nosrati, \emph{$r$-Strong edge colorings of graphs}, Discrete Mathematics, \textbf{306} (2006) 3005--3010. \vspace{-0.25cm}

%\bibitem{}

\bibitem{zhang-2} Z.F. Zhang, J.W. Li, X.E. Chen, et al., \emph{$D(\beta)$--vertex-distinguishing total coloring of graphs}, Science China Mathematics, \textbf{49}(2006), no. 10, 1430--1440.   \vspace{-0.25cm}

\bibitem{zhang-3} Z.F. Zhang, H. Cheng, B. Yao, et al., \emph{On the adjacent-vertex-strongly-distinguishing total coloring of graphs}, Science in China, \textbf{51}(2008), no. 3, 427--436.   \vspace{-0.25cm}

\bibitem{zhang-1} Z.F. Zhang, X.E. Chen, J.W. Li, \emph{On the adjacent vertex distinguishing total coloring of graphs}, Sci. China Ser. A \textbf{48}(2005), no. 3, 289--299.   \vspace{-0.25cm}

\bibitem{wen} F. Wen, L.N. Huang, M.C. Li, B. Yao, \emph{Relations of adjacent strong
chromatic indexes between graphs and their subgraphs}, Ars Combinatoria, 2020, in press. \vspace{-0.25cm}

\bibitem{Liang} P. Liang, Elimination methods and comparison of co-frequency interference in LTE system [C], Papers of the 2009 Annual Meeting of the Information and Communication Network Technology Committee of the Chinese Communication Society, 2009, \textbf{2}.\vspace{-0.25cm}

\bibitem{winter} J.H. Winters, \emph{Optimum combining in digital mobile radio with cochannel interference}, IEEE Transactions on Vehicular Technology, 2013, \textbf{33}(3):144-155.\vspace{-0.25cm}

}
\end{thebibliography}
\end{document}